\documentclass[letterpaper, 10 pt, conference]{ieeeconf} 
\IEEEoverridecommandlockouts  
\usepackage{amsmath,amsfonts,amssymb}
\usepackage{algorithmic}
\usepackage{algorithm}
\usepackage{array}
\usepackage{xcolor}
\usepackage{flushend}
\usepackage{textcomp}
\usepackage{stfloats}
\usepackage{url}
\usepackage{verbatim}
\usepackage{graphicx}
\usepackage{caption}
\usepackage{subcaption}
\usepackage{multirow}
\usepackage{cite}
\newtheorem{theorem}{Theorem}
\newtheorem{definition}{Definition}
\newtheorem{lemma}{Lemma}\let\labelindent\relax
\usepackage{enumitem}

\begin{document}

\title{A Fixed-Time Sliding-Mode Framework for Constraint Optimization}

\author{Baby Diana, Priyanka Singh, Shyam Kamal, Sandip Ghosh and Bijnan Bandyopadhyay
\thanks{This work was supported by the project titled ``Center for Development of Drone Related Technologies (CDRT)", IIT (BHU) Varanasi, India under grant (No. R\&D/SA/I-DAPTIIT(BHU)/EE/22-23/08/448)}
\thanks{Baby Diana, Priyanka Singh, Shyam Kamal and Sandip Ghosh are with the Department of Electrical Engineering,  Indian Institute of Technology (BHU) Varanasi, India. 
        {\tt\small babydiana.rs.eee21@iitbhu.ac.in, priyankasingh.rs.eee21@itbhu.ac.in, shyamkamal.eee@iitbhu.ac.in, sghosh.eee@iitbhu.ac.in }}
        \thanks{Bijnan Bandyopadhyay is with the Department of Electrical Engineering,
Indian Institute of Technology, Jodhpur, Rajasthan 342030, India. {\tt\small bijnan@iitj.ac.in }}
}

\maketitle

\begin{abstract}
This paper develops a robust fixed-time optimization framework for constrained problems that guarantees exact constraint satisfaction and convergence to Karush–Kuhn–Tucker (KKT) points within a fixed time, independent of initial conditions. The approach treats the Lagrange multipliers as control inputs, composed of an equivalent control and a switching control, with the system states representing the decision variables. An equivalent control steers the gradient flow to a local KKT point asymptotically for nonconvex objectives and to the unique global optimum in fixed time for convex objectives. Constraint enforcement is achieved by embedding the equality constraints directly as a sliding manifold, with a fixed-time switching control ensuring rapid and reliable feasibility. The framework further accounts for matched disturbances, providing robustness guarantees that are theoretically characterized and illustrated using spherical constraints. Numerical studies on a 3-bus AC optimal power flow problem and a distributed consensus-based parameter estimation problem demonstrate the effectiveness, scalability, and robustness of the proposed approach.
\end{abstract}


\section{Introduction}
Continuous-time optimization models algorithms as differential equations, providing insights into stability and convergence for large-scale convex and nonconvex problems. Gradient flow is the continuous analogue of gradient descent, while proximal methods arise from backward Euler discretization. Feasibility can be enforced by combining projected gradient and Gauss–Newton directions, and inequality-constrained problems incorporate KKT conditions to ensure constraint satisfaction \cite{schropp2000dynamical,zhou2007convergence}.

Extensive research on continuous-time optimization has focused on Lagrange multiplier methods, particularly primal-dual gradient dynamics (PDGD) \cite{arrow1958studies}, with exponential stability results for strongly convex and smooth objectives. Most works analyze existing algorithms rather than designing new ones. Exceptions include \cite{allibhoy2023control}, which uses control barrier functions for feasibility and stability, and \cite{feng2020dynamical}, which models constrained optimization as a perturbed projected dynamical system. More recently, \cite{cerone2025new} proposed the Controlled Multipliers Optimization (CMO) framework, treating Lagrange multipliers as control inputs and enabling systematic design of first-order algorithms with convergence guarantees for both convex and nonconvex equality-constrained problems.

Fixed-time gradient flows for constrained optimization have been studied in \cite{OJ24,garg_fixed}, with convergence to KKT points for strongly convex problems, and extended in \cite{diana2025finite} using CBFs. However, these works do not consider external disturbances, which can affect feasibility and convergence. This paper addresses this gap by proposing a robust fixed-time optimization framework, combining an equivalent control $u_{eq}$ term \cite{utkin2013sliding,kamal2026robust} to drive the KKT point (asymptotically or in fixed time) and a sliding term $u_{sw}$ \cite{moulay_fxt,kamal2026robust} to enforce constraint satisfaction in fixed time. Matched disturbances are included to capture practical implementation uncertainties such as noise, discretization, and modeling errors. The main contributions of this work are summarized as follows:
\begin{enumerate}
    \item By treating the equality constraints as a sliding surface, a control law is designed to ensure fixed-time convergence to constraint satisfaction, independent of initial conditions.
    \item On the sliding manifold, the reduced-order dynamics ensure local asymptotic convergence to a KKT point for nonconvex objectives, while for $\mu$-strongly convex objectives the primal dynamics converge to the unique global optimizer in fixed time.
    \item The proposed control framework ensures fixed-time reachability of the constraint manifold and convergence even in the presence of matched disturbances $\xi(x)$, as demonstrated for nonconvex problems with spherical constraints.
    \item \textit{Numerical validation:} We consider the three-bus AC power system and a distributed estimation problem to demonstrate applicability in physical systems with parameter uncertainties (e.g., load variations, modeling mismatch) and networked settings with communication noise and delays, highlighting scalability and robustness where centralized solutions are impractical, beyond discretization effects.
\begin{itemize}
    \item AC-OPF (3-bus): Active power minimization under power balance constraints with fixed-time convergence to a KKT point.
    \item Distributed estimation: Consensus-based convex problem ensuring fixed-time feasibility and convergence to the KKT solution.
\end{itemize}
\end{enumerate}
The paper is organized as follows: Section II presents notations and preliminaries; Section III formulates the problem; Section IV develops the proposed SMC-based optimization framework with finite-time and robustness analysis; Section V provides application results; and Section VI concludes the paper.

\section{Notations and Mathematical Preliminaries}
\subsection{Notations:}
Let $\mathbb{R}^n$ denote the set of real column vectors of dimension $n$. The $p$-norm $\|x\|_p$ for $p \in \{1,2,\infty\}$ is defined by: $\|x\|_1 = \sum |x_i|$, $\|x\|_2$ (Euclidean norm), and $\|x\|_\infty = \max |x_i|$. The absolute value of a scalar is denoted by $|\cdot|$. For a matrix $A \in \mathbb{R}^{m \times n}$, the range is $\mathrm{range}(A) = \{Ax\}$, and the kernel is $\ker(A) = \{x : Ax = 0\}$. The Jacobian of $h:\mathbb{R}^n\to\mathbb{R}^m$ is $J_h(x)\in\mathbb{R}^{m\times n}$ with $(i,j)$-th entry $\frac{\partial h_i(x)}{\partial x_j}$. Given $g:\mathbb{R}^n\to\mathbb{R}$, the gradient denotes as $\nabla g$ and $\nabla^2_{xx} g(x)\in\mathbb{R}^{n\times n}$ denotes its Hessian matrix of second-order partial derivatives with respect to $x$. The sign function is $\operatorname{sign}(x) = x/|x|$ for $x \neq 0$. $A \circ B$ means elementwise multiplication. For a function $\zeta:\mathcal{X}\to\mathcal{Y}$ with $\mathcal{X}\subseteq\mathbb{R}^n$ and $\mathcal{Y}\subseteq\mathbb{R}^m$, $\zeta \in \mathcal{C}^k(\mathcal{X},\mathcal{Y})$ denotes that $\zeta$ is $k$-times continuously differentiable. 
\subsection{Mathematical Preliminaries} 
Consider the dynamical system described by
\begin{equation}\label{1}
    \dot{x} = f(x(t)),
\end{equation}
where $x \in \mathbb{R}^n$ is the state vector, $f: \mathbb{R}^n \to  \mathbb{R}^n$ is a continuous nonlinear vector field with equilibrium at the origin $f(0)=0$, $x_0$ is the initial state, and $t_0 \ge 0$ is the initial time; these notions are used in Section IV for the Lyapunov-based analysis of the proposed method.

\begin{definition}[Fixed-time stability(FxTS)] \cite{polyakov_fix}
The origin is globally fixed-time stable if it is finite-time stable and there exists a $T_M>0$ which is uniform w.r.t $(t_0,x_0)$ such that settling time $T(t_0,x_0)\le T_M$ for all $x_0 \in \mathbb{R}^n$.
\end{definition}

\begin{lemma}{\cite{polyakov_fix}}
Consider system \eqref{1}. If there exists $V \in \mathcal{C}^1(\mathbb{R}^n,\mathbb{R}_{\ge 0})$ with $\psi_1(x) \le V(x) \le \psi_2(x)$, $V(0)=0$, and $\dot V(x) \le -\rho_1 V^{\beta_1}(x) - \rho_2 V^{\beta_2}(x)$, where $\rho_{1,2}>0$, $\beta_1 \in (0,1)$, $\beta_2>1$, then the origin is fixed-time stable with settling time $T(t_0,x_0) \le T_M = \frac{1}{\rho_1(1-\beta_1)} + \frac{1}{\rho_2(\beta_2-1)}$.
\end{lemma}

Consider the optimization problem
\begin{equation}\label{eq:opt_problem}
\min_{x \in \mathbb{R}^n} \; \phi(x),
\end{equation}
where $\phi \in \mathcal{C}^2(\mathbb{R}^n,\mathbb{R})$ is a convex function. A point $x^\star$ is a global minimizer if and only if $\nabla \phi(x^\star)=0$, and it is unique when $\phi$ is strictly convex \cite{boyd_optmz}.
\begin{definition}[$\mu$-Strong Convexity ] \cite{beck_convx}
A function $\phi \in \mathcal{C}^1(\mathbb{R}^n,\mathbb{R})$ is $\mu$-strongly convex $(\mu>0)$ if $\phi(y) \ge \phi(x) + \nabla \phi(x)^\top (y-x) + \frac{\mu}{2}\|y-x\|^2,
\quad \forall x,y \in \mathbb{R}^n$.
\end{definition}

\begin{lemma} [Lojasiewicz’s Inequality]\cite{l_ineq}
If $ \phi \in \mathcal{C}^1(\mathbb{R}^n, \mathbb{R}) $ is real analytic near $ y \in \mathbb{R}^n $, then there exist constants $ \mu_f > 0 $ and $ \beta \in [0,1) $ such that  
$
\|\nabla \phi(x)\| \geq \mu_f |\phi(x) - \phi(y)|^\beta
$
holds in a neighbourhood of $ y $.
\end{lemma}

\section{Problem Formulation}
Consider an optimization problem with equality constraints:
\begin{equation}
\min_{x \in \mathbb{R}^n} \phi(x) \quad \text{subject to} \quad h(x) = 0,
\label{eq:op}
\end{equation}
where $\phi \in  \mathcal{C}^n (\mathbb{R}^n,\mathbb{R})$ and $h \in \mathcal{C}^n (\mathbb{R}^n,\mathbb{R}^m)$ are possibly non-convex functions. The Lagrangian is defined as:
$\mathcal{L}(x, \lambda) = \phi(x) + \lambda^\top h(x)$,
with Lagrange multipliers \( \lambda \in \mathbb{R}^m \).
\begin{lemma}\cite{luenberger1984linear} \label{Fonc and SONC}
For optimization problem \eqref{eq:op} with $h$ satisfying the constraint qualification $\mathrm{rank}(J_h(x^*))=m$, any local minimizer $x^*$ admits a Lagrange multiplier $\lambda^*\in\mathbb{R}^m$ such that the first-order necessary condition $$\nabla \phi(x^*) + J_h(x^*)^\top \lambda^* = 0, \qquad h(x^*)=0,$$ holds. Moreover, since first-order conditions alone characterize stationary points that may include saddle points or maxima in nonconvex problems, a second-order necessary condition must also be satisfied, namely $$d^\top \nabla^2_{xx}\mathcal{L}(x^*,\lambda^*) d \ge 0 \quad \forall d \in \ker(J_h(x^*)).$$ These conditions follow from the Karush--Kuhn--Tucker(KKT) framework, which generalizes the classical method of Lagrange multipliers. 
\end{lemma}
Consider the constrained optimization problem \eqref{eq:op}, motivated by \cite{cerone2025new} a dynamical systems interpretation of the stationarity conditions Lemma \ref{Fonc and SONC}:
\begin{equation}
\mathcal{H}_{FxTS}: 
\begin{cases}
\dot{x}(t) = -\nabla \phi(x(t)) - J_h(x(t))^\top \lambda_{FxTS}(x(t)), \label{eq: gf dynamics fx} \\
y(t) = h(x(t)),
\end{cases}
\end{equation}
where $\lambda_{FxTS}(x(t))$ is treated as a control input and $y(t)$ represents the regulated output (constraint manifold). We assume $J_h(x)$ is full rank. This control perspective motivates a feedback design for $\lambda_{FxTS}(x(t))$ to enforce the KKT conditions. An equilibrium $(x^*,\lambda^*)$ of $\mathcal{H}_{FxTS}$ is stationary if and only if $h(x^*)=0$. \\
\textbf{Problem Statement: }
Design a feedback control law $\lambda_{FxTS}(x)$ such that the resulting closed-loop system satisfies the following properties:
\begin{enumerate}
    \item \emph{Fixed-time constraint satisfaction:}  
    The constraint violation dynamics admit a sliding manifold $y=0$ that is reached in fixed time, i.e., $ \exists\, T_c > 0 \ \text{independent of initial conditions such that} \ 
    y(t)=0 \ \forall t \ge T_c$, using a fixed-time sliding mode control strategy in the sense of \cite{moulay_fxt}.

   \item \emph{Optimality:} On the sliding manifold, the reduced-order dynamics ensure asymptotic convergence $x(t)\to x^\star$ as $t\to\infty$ for nonconvex problems (KKT point), while for $\mu$-strongly convex objectives, there exists $T_c>0$ independent of initial conditions such that $x(t)=x^\star$ for all $t\ge T_c$ (unique global optimizer).
   
    \item \emph{Robustness to matched disturbances \cite{kamal2026robust}:}  
    For matched disturbances $\xi(t,x)=J_h^\top\eta(t,x)$ with $\|\eta(t,x)\|\le\bar{\eta}$, the modified feedback law preserves fixed-time reachability of the sliding manifold and guarantees the optimality convergence of the closed-loop dynamics $\dot{x}=-\nabla\phi(x)-J_h(x)^\top\lambda_{FxTS}+\xi(t,x)$.
\end{enumerate}

\section{Main Result}
\subsection{Fixed-time Sliding Mode control for Nonconvex constraint optimization}
Consider  the closed-loop system \eqref{eq: gf dynamics fx}. The sliding surface is constructed based on the output to enforce the constraint, and is given by:
\begin{equation}\label{eq: SF1}
\mathcal{S}(x) = \{x \in \mathbb{R}^n \mid h(x) = 0\}.
\end{equation}
The control objective is to drive the sliding variable $s(x)=h(x)$ to zero in fixed time and to steer the closed-loop system toward a stationary point $(x^\star,\lambda^\star)$ satisfying the first-order necessary optimality conditions of Lemma~\ref{Fonc and SONC}. To enforce fixed-time convergence of the  constraint satisfactions, i.e., $h(x)\to 0$ for all $t \geq T_c$, we propose the following discontinuous fixed-time feedback law:
\begin{align}
    \label{eq:lambda_fxt}
\lambda_{\mathrm{FxTS}}(x)&= G(x)^{-1}(- J_h(x)\nabla \phi(x) \nonumber\\
&+\!\alpha \circ \mathrm{sgn}(h(x))\! \circ \!|h(x)|^{p} + \beta \circ \mathrm{sgn}(h(x)) \!\circ\! |h(x)|^{q}),
\end{align} 
where $G(x)=J_h(x)J_h(x)^\top \in \mathbb{R}^{m\times m}$, $\alpha = \operatorname{diag}(\alpha_1,\dots,\alpha_n), \quad \beta = \operatorname{diag}(\beta_1,\dots,\beta_n), (\alpha_i,\beta_i)>0$, $0<p<1$, and $q>1$. Though Eq. \eqref{eq:lambda_fxt} has a diagonal SMC-like structure via element-wise action on $h(x)$, it differs by including fixed-time gains and the coupling term $J_h(x)\nabla \phi(x)$, forming a unified fixed-time optimization–control framework.
\begin{theorem}
Consider system \eqref{eq: gf dynamics fx} with control input $\lambda_{\mathrm{FxTS}}$ defined in \eqref{eq:lambda_fxt}, decomposed into $\lambda_{\mathrm{eq}}(x) = -G(x)^{-1}J_h(x)\nabla \phi(x)$ and $\lambda_{\mathrm{sw}}(x) = G(x)^{-1}\!\left[\alpha \circ \mathrm{sgn}(h(x)) \circ |h(x)|^{p} + \beta \circ \mathrm{sgn}(h(x)) \circ |h(x)|^{q}\right]$, with $\alpha_i,\beta_i>0$, $0<p<1$, and $q>1$; then the state trajectory reaches the constraint manifold $\mathcal{S}$ in fixed time, i.e., there exists $T_c>0$, independent of the initial condition, such that $h(x)=0$ for all $t\ge T_c$, with settling time satisfying $T_c \le T_{\max}= \frac{2}{\alpha_{\min}(1-p)} + \frac{2}{\beta_{\min}(q-1)}$, where $\alpha_{\min}=\min_i \alpha_i$ and $\beta_{\min}=\min_i \beta_i$.
\end{theorem}

\begin{proof}
The proof is divided into two parts \footnote{Due to page limitations, the equation is presented in a more concise form. For clarity, dependence on $t$ and $x$ is omitted when it is obvious.}.

\emph{Part~1 (Sliding condition and equivalent control): }
On the manifold $\mathcal{S}$, invariance requires $\dot{h}(x)=0$, i.e., $J_h(x)\dot{x}=0.$
Substituting \eqref{eq: gf dynamics fx} yields
$J_h(x)\left(-\nabla \phi(x)-J_h(x)^\top\lambda_{\mathrm{eq}}(x)\right)=0$,
which implies $G(x)\lambda_{\mathrm{eq}}(x)=-J_h(x)\nabla \phi(x)$.
Since $G(x)$ is invertible by the full-row-rank assumption, the equivalent
control is uniquely given by $\lambda_{\mathrm{eq}}(x)=-G(x)^{-1}J_h(x)\nabla \phi(x)$.

\emph{Part~2 (Fixed-time reachability): }
Sliding surface is defined as $s=h(x)$ and consider the Lyapunov function $V=\frac12\|s\|_2^2$.
Differentiating along the closed-loop trajectories gives
\[
\dot{V}
= s^\top J_h(x)\dot{x}
= -s^\top\!\left[
\alpha \circ \mathrm{sgn}(s)\! \circ \! |s|^{p}
+ \beta \circ \mathrm{sgn}(s)\! \circ \! |s|^{q}
\right].
\]
Thus, 
\begin{align*}
    \dot{V}&= -\sum_{i=1}^m(\alpha_i |s_i|^{1+p}+ \beta_i |s_i|^{1+q}) \\
    &\le -\alpha_{\min}\|s\|_2^{1+p}-\beta_{\min}\|s\|_2^{1+q}.
\end{align*}
Using $\|s\|_2=\sqrt{2V}$, we obtain $\dot{V} \le -c_1 V^{\frac{1+p}{2}} - c_2 V^{\frac{1+q}{2}}$, where $c_1=\alpha_{\min}2^{\frac{1+p}{2}}$ and $c_2=\beta_{\min}2^{\frac{1+q}{2}}$.
Since $0<\frac{1+p}{2}<1$ and $\frac{1+q}{2}>1$, standard fixed-time
stability results imply convergence of $V$ to zero in fixed time, with
\begin{align}
 T_c \le \!\!\frac{1}{c_1\!\left(1-\frac{1+p}{2}\right)}
     +\!\frac{1}{c_2\!\left(\frac{1+q}{2}-1\right)} \nonumber\\
\le \!\frac{2}{\alpha_{\min}(1-p)} +\! \frac{2}{\beta_{\min}(q-1)}. \nonumber  
\end{align}
Hence, $h(x)=0$ for all $t\ge T_c$, completing the proof.
\end{proof}
\begin{lemma}[Asymptotic Optimality on the Sliding Surface]
Consider the closed loop system \eqref{eq: gf dynamics fx} with sliding surface \eqref{eq: SF1}.
Under the control input $\lambda_{FxTS}$ \eqref{eq:lambda_fxt} the state trajectory reaches sliding surface  in time $T_c$, then for all $t\ge T_c$ the closed-loop system evolves as
\[
\dot{x}=-P(x)\nabla \phi(x),
\]
where $P(x)=I-J_h(x)^\top G(x)^{-1}J_h(x)$. The projected gradient flow is asymptotically stable and converges to an optimal point $x^\star$. Along trajectories on the constraint manifold $\mathcal{S}$, the objective function $\phi(x)$ is non-increasing, and every accumulation point $x^\star$ satisfies the KKT conditions
\[
h(x^\star)=0,\qquad \exists,\lambda^\star \ \text{such that}\ \nabla \phi(x^\star)+J_h(x^\star)^\top\lambda^\star=0.
\]
If $x^ \star$ is a strict local minimum satisfying the second-order sufficient conditions, then it is locally exponentially stable on $\mathcal{S}$.
\end{lemma}

\begin{proof}
Once the trajectory is reached on sliding manifold $\mathcal{S}$ in fixed time $T_c$, the
dynamics satisfy $J_h(x)\dot x=0$. Substituting the equivalent control
$\lambda_{\mathrm{eq}}$, the reduced dynamics becomes $$\dot x=-P(x)\nabla \phi(x)$$ with
$P(x)=I-J_h(x)^\top(J_hJ_h^\top)^{-1}J_h(x)$. For $V(x)=\phi(x)-\phi(x^*)$,
$\dot V=-\nabla \phi(x)^\top P(x)\nabla \phi(x)\le0$, and equality holds iff the
KKT conditions are satisfied. Hence, trajectories converge to the set of KKT
points, which are locally exponentially stable if the second-order sufficient
conditions hold.
\end{proof}
KKT convergence follows from the equivalent sliding dynamics enforcing stationarity, while fixed-time feasibility guarantees constraint satisfaction under standard regularity assumptions.
\subsection{Robust Fixed-Time Reachability Under Matched Disturbances}
In this section, we extend the proposed framework to systems with matched disturbances and develop a fixed-time sliding mode control strategy for constraint enforcement. Despite the presence of disturbances, the closed-loop system remains asymptotically stable, while reachability of the constraint manifold is guaranteed within a fixed time $T_c$. The disturbed system is described by
\begin{equation}
\label{eq:robust_fx_time}
\dot{x} = -\nabla \phi(x) - J_h(x)^\top \lambda_{FxTS} + J_h^\top \eta(t,x),
\end{equation}
where the disturbance is matched with the control input and satisfies $\|\eta(t,x)\| \le \bar{\eta}$, for some known constant $\bar{\eta} > 0$.

\begin{theorem}
Consider the perturbed system \eqref{eq:robust_fx_time}. Assume $\mathrm{rank}(J_h(x))=m$ and let the control law be $\lambda = \lambda_{\mathrm{FxTS}}(x)$, with gains satisfying $\alpha_i > \bar{\eta}$ and $\beta_i>0$. Then, the sliding manifold $\mathcal{S}$ give by \eqref{eq: SF1} is reached in time $T_c \le \frac{2}{(\alpha_{\min}-\bar{\eta})(1-p)} + \frac{2}{\beta_{\min}(q-1)}$, independently of the initial condition. The effect of the disturbance is dominated by the discontinuous SMC term, preserving fixed-time reachability of the sliding surface.
\end{theorem}

\begin{proof}
Let $V=\tfrac12\|h(x)\|_2^2$. Along trajectories,
\[
\dot V = h^\top J_h \dot x
       = -h^\top\!\left[\alpha\!\circ\!\mathrm{sgn}(h)\!\circ\!|h|^p
       + \beta\!\circ\!\mathrm{sgn}(h)\!\circ\!|h|^q\right]
       + h^\top \eta.
\]
Using $\|h^\top\eta\|\le \|h\|\bar{\eta}$ and $\alpha_i>\bar{\eta}$ yields $\dot V \le -(\alpha_{\min}-\bar{\eta})\|h\|^{1+p}-\beta_{\min}\|h\|^{1+q}$, which implies fixed-time reachability by standard Lyapunov arguments. On the sliding manifold $h(x)=0$, the disturbance term $J_h^\top\eta$ is handled by the control $\lambda_{FxTS}$, yielding the nominal projected gradient dynamics. Hence, convergence to the KKT set follows from the disturbance-free analysis as for Lemma 4. 
\end{proof}
For a more detailed analysis, the reader is referred to the authors’ main paper~\cite{kamal2026robust}.\\
\begin{figure}[h!]
    \centering
    \includegraphics[width=1.05\linewidth]{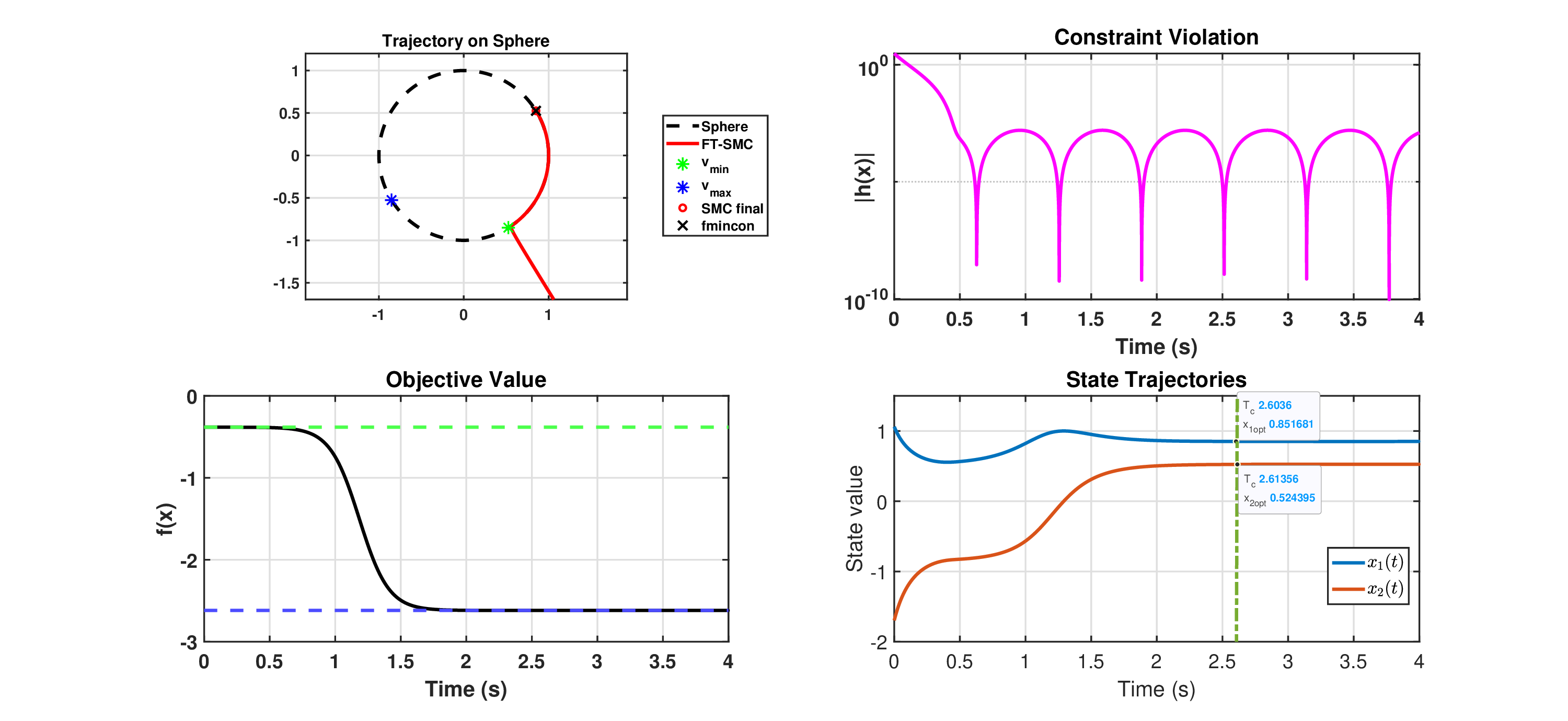}
    \caption{This figure shows fixed-time SMC design for Example 1 \eqref{eq:sphere_problem}. The fixed-time gains in \eqref{eq:lambda_fxt} are chosen as $\alpha=5$, $\beta=5$, $p=0.5$, and $q=1.5$. The proposed fixed-time SMC dynamics are validated against \textit{fmincon}, which converges to the optimal solution $x^{*} = [0.85,;0.52]^{\top}$ within $T_c \le 2.64$. The optimal objective value $\phi(x^*) = -2.61$ is achieved by both \textit{fmincon} and the proposed method under the matched disturbance $\delta = 0.05$ with $\eta(t)=\delta\sin(5t)$.}
    \label{fig: ex1}
\end{figure}

\textit{Example 1 (Sphere-Constrained Nonconvex Optimization):}
    Consider the nonconvex equality-constrained optimization problem
\begin{equation}
\label{eq:sphere_problem}
\min_{x\in\mathbb{R}^2} \phi(x)
\quad \text{s.t.} \quad
h(x)=\|x\|^2-1=0,
\end{equation}
The feasible set is the unit sphere $\mathcal{S}=\{x \in \mathbb{R}^2 :\|x\|=1\}$, with constraint Jacobian
$J_h(x)=2x^\top$ and Gram matrix $G(x)=4\|x\|^2$. We consider the control-oriented dynamics
\begin{equation}
\label{eq:sphere_dyn}
\dot{x}=-\nabla \phi(x)-2x\lambda(x),
\end{equation}
where $\lambda(x)$ is treated as a scalar control input. Defining the sliding
variable $s(x)=h(x)$, fixed-time constraint enforcement is achieved by
\begin{equation}
\label{eq:sphere_lambda}
\lambda_{\mathrm{FxTS}}(x)=\!\!\frac{1}{4\|x\|^2}\Big(\!\!-2x^\top\nabla \phi(x)+\alpha\,\mathrm{sgn}(s)|s|^p+\beta\,\mathrm{sgn}(s)|s|^q\!\Big),
\end{equation}
with $\alpha=5$, $\beta=5$, $p=0.5$, and $q=1.5$. The resulting sliding dynamics satisfy $\dot{s}=-\alpha\,\mathrm{sgn}(s)|s|^p-\beta\,\mathrm{sgn}(s)|s|^q$, which guarantees fixed-time convergence $s(x)\to 0$ for all initial conditions, with an upper bound $T_c \le \frac{1}{\alpha(1-p)}+\frac{1}{\beta(q-1)}$. Hence, $\|x\|=1$ for all $t\ge T_c$. On the manifold $\mathcal{S}$, the equivalent control is $\lambda_{\mathrm{eq}}=-\frac{1}{2}x^\top\nabla \phi(x)$, yielding the reduced dynamics
\begin{equation}
\label{eq:proj_grad}
\dot{x}=-(I-xx^\top)\nabla \phi(x),
\end{equation}
which corresponds to the projected gradient flow on $\mathcal{S}$. Along
\eqref{eq:proj_grad}, $\frac{\operatorname{d}}{\operatorname{dt}}\phi(x)=-\nabla \phi(x)^\top (I-xx^\top)\nabla \phi(x)\le 0$, with equality if and only if $\nabla \phi(x)=\lambda x$. Consequently, trajectories converge asymptotically to stationary (KKT) points of \eqref{eq:sphere_problem}. If the reduced Hessian
$(I-xx^\top)\nabla^2\phi(x)(I-xx^\top)$ is positive definite at a stationary point, local asymptotic stability is ensured. In the presence of matched disturbances
\begin{equation}
  \dot{x}=-\nabla \phi(x)-2x\lambda+2x\eta(t,x), \quad |\eta(t,x)|\le\bar{\eta},  
\end{equation}
robust fixed-time constraint satisfaction is preserved by augmenting
\eqref{eq:sphere_lambda} with $\rho\,\mathrm{sgn}(s)/(4\|x\|^2)$, where
$\rho>4\|x\|^2\bar{\eta}$, ensuring fixed-time reachability of the constraint
manifold and practical convergence to stationary points as shown in Fig. \ref{fig: ex1}.

\subsection{Fixed-time convergence of Convex optimization problem}
In this section, we consider convex optimization problems to establish fixed-time convergence to the optimal solution. The nonconvex case is not addressed, as fixed-time guarantees generally require additional structure. Using $\mu$-strong convexity, explicit convergence bounds are derived. Unlike primal–dual methods (e.g., \cite{OJ24,garg_fixed}), no dual dynamics are used; constraints are enforced via a fixed-time SMC mechanism. After reaching the constraint manifold, convergence reduces to a projected gradient flow, yielding fixed-time convergence to the KKT point in $T_c$. Specifically, for $\phi,h\in C^2$ with $\mathrm{rank}(J_h(x))=m$ on $\mathcal{S}=\{x:h(x)=0\}$, the closed-loop system \eqref{eq: gf dynamics fx} is modified as follows.
\begin{equation} \label{eq:fxgf_smc}
 \dot{x}=\!\!-F_1-\!J_h(x)^\top\! \Lambda_{FxTS}(x),   
\end{equation}
where $F_1=\gamma_1\frac{\nabla \phi(x)}{\|\nabla \phi(x)\|^{1-r_1}}+\gamma_2\frac{\nabla \phi(x)}{\|\nabla \phi(x)\|^{1+r_2}},~\gamma_1, \gamma_2>0$ $r_1 \in (0,1)$ and $r_2>1$. The fixed-time SMC multiplier $\Lambda_{FxTS}=\Lambda_{eq}+\Lambda_{sw}$ is defined as:
\begin{equation} \label{eq: fxts_1}
\begin{aligned}
\Lambda_{FxTS}
&= G(x)^{-1}\Big[-J_h(x)F_1 +\alpha \circ \operatorname{sgn}(h(x)) \circ |h(x)|^{p}\\
&\quad+ \beta \circ \operatorname{sgn}(h(x)) \circ |h(x)|^{q}
\Big].
\end{aligned}
\end{equation}
\begin{theorem}
Consider the closed system \eqref{eq:fxgf_smc} with control law \eqref{eq: fxts_1}, where 
$\alpha_i>0$, $\beta_i>0$, $0<p<1$, $q>1$, and $G(x)=J_h(x)J_h(x)^\top$ is positive definite.
Then:
\begin{enumerate}
    \item The constraint manifold $\mathcal{S}=\{x\mid h(x)=0\}$ is reached in fixed time $T_c \le \frac{2}{\alpha_{\min}(1-p)} + \frac{2}{\beta_{\min}(q-1)}$.
   
    \item If $\phi$ is $\mu$-strongly convex, then for $t\ge T_c$, the system converges to 
    the unique KKT point $x^\star$ in fixed time $T_o \le \frac{2}{\gamma_1(2\mu)^{\frac{1+r_1}{2}}(1-r_1)} + \frac{2}{\gamma_2(2\mu)^{\frac{1+r_2}{2}}(r_2-1)}$.
\end{enumerate}
The total convergence time $T_{total}=T_c+T_o$ is independent of $x(0)$.
\end{theorem}
\begin{proof}
\begin{enumerate}
    \item Using $V_c=\frac{1}{2}\|h\|^2$, we obtain $\dot{V}_c \leq -\kappa_1 V_c^{\frac{p+1}{2}} - \kappa_2 V_c^{\frac{q+1}{2}}$ with $\kappa_1=2^{\frac{p+1}{2}}\alpha_{\min}$, $\kappa_2=2^{\frac{q+1}{2}}\beta_{\min}$. By Lemma~1 (fixed-time stability), $h(t)=0$ for all $t\geq T_c$ with bound as given.

    \item For $t \geq T_c$, $h(x(t)) \equiv 0$ and \eqref{eq:fxgf_smc} reduces to the projected fixed-time gradient flow
\[
\dot{x} = -P_\perp(x)F_1,
\]
where $P_\perp = I - J_h^\top(J_h J_h^\top)^{-1}J_h$ projects onto $\ker J_h$. Let $V_o = \phi - \phi(x^\star)$. By $\mu$-strong convexity and Lojasiewicz’s Inequality, $\|\nabla\phi\| \geq \sqrt{2\mu V_o}$. Then $\dot{V}_o \leq -\gamma_1(2\mu V_o)^{\frac{r_1}{2}} - \gamma_2(2\mu V_o)^{\frac{r_2}{2}}$, which gives the fixed-time bound $T_o \leq \frac{2}{\gamma_1(2\mu)^{\frac{r_1}{2}}(1-r_1)} + \frac{2}{\gamma_2(2\mu)^{\frac{r_2}{2}}(r_2-1)}$.
\end{enumerate}
Both $T_c$ and $T_o$ are independent of $x(0)$; hence $T_{\mathrm{total}} = T_c + T_o$ is also independent of initial conditions.
\end{proof}
\section{Applications}
\subsection{Fixed-Time Sliding-Mode Gradient Flow Formulation for a Three-Bus AC--OPF}
Consider a three-bus AC power system consisting of one slack bus, one load bus, and one generator bus. The objective is to minimize the active power generation cost at the generator bus while satisfying the nonlinear AC power flow with equality constraints. The cost function is chosen as a convex quadratic function of the generated active power:
\begin{equation}
\phi(x) = a P_{g3}^{2} + b P_{g3},
\label{eq:cost}
\end{equation}
where $P_{g3}$ denotes the active power generated at Bus~3, and $a>0$, $b\ge 0$ are known cost coefficients.
Let the set of buses be $\mathcal{N}=\{1,2,3\}$.  
Bus~1 is designated as the slack bus with fixed voltage magnitude and angle $(V_1,\theta_1)=(1,0)$.
Bus~2 is a PQ (load) bus, and Bus~3 is a PV (generator) bus in Fig. \ref{fig:3 bus system}. The state vector is defined as $x =
\begin{bmatrix}
\theta_2 & \theta_3 & V_2 & V_3 & P_{g3} & Q_{g3}
\end{bmatrix}^{\top}
\in \mathbb{R}^{6}$.
The active and reactive power injections at bus $i\in\mathcal{N}$ are given by the standard AC power flow equations
\begin{align}
P_i(x) &=
\sum_{j=1}^{3} V_i V_j
\big(
G_{ij}\cos(\theta_i-\theta_j)
+
B_{ij}\sin(\theta_i-\theta_j)
\big),  \nonumber\\
Q_i(x) &=
\sum_{j=1}^{3} V_i V_j
\big(
G_{ij}\sin(\theta_i-\theta_j)
-
B_{ij}\cos(\theta_i-\theta_j)
\big),
\end{align}
where $G_{ij}$ and $B_{ij}$ denote the conductance and susceptance of the network admittance matrix. The AC power balance constraints are compactly written as 
$h(x) =
\begin{bmatrix}
P_2(x) + P_{d2} \\
Q_2(x) + Q_{d2} \\
P_3(x) - P_{g3} \\
Q_3(x) - Q_{g3}
\end{bmatrix}
= 0,$
where $P_{di}$ and $Q_{di}$ denote the active and reactive load demands. For the three-bus AC optimal power flow case study (Fig. \ref{fig:3 bus system}), the active and reactive load demands are given by $P_d=[0,\,0.6,\,0.4]^\top$ and $Q_d=[0,\,0.3,\,0.2]^\top$, respectively, with a susceptance matrix $B=\begin{bmatrix}
0 & -10 & -5\\
-10 & 0 & -8\\
-5 & -8 & 0
\end{bmatrix}$,
and a single generator located at bus~1. The quadratic generation cost in~\eqref{eq:cost} is parameterized by $a=0.2$ and $b=1$.
\begin{figure}[h!]
    \centering
    \begin{subfigure}{0.48\linewidth}
        \centering
        \includegraphics[width=0.9\linewidth]{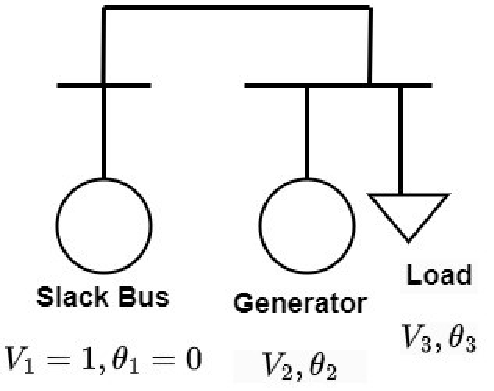}
        \caption{3-Bus system}
        \label{fig:3 bus system}
    \end{subfigure}
    \hfill
    \begin{subfigure}{0.48\linewidth}
        \centering
        \includegraphics[width=\linewidth]{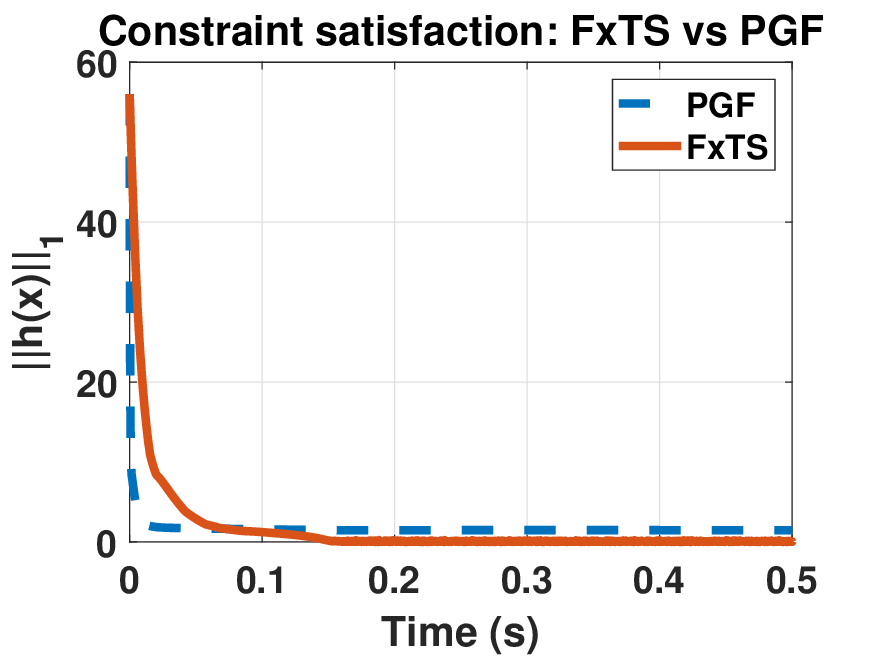}
        \caption{Constraint evolution}
        \label{fig:constraint}
    \end{subfigure}
    \caption{3-Bus System and constraint behavior}
    \label{fig:acopf}
\end{figure}
In the sliding-mode framework, the constraint vector defines the sliding manifold
$\mathcal{S}=\{x\in\mathbb{R}^{6}\mid h(x)=0\}$. The fixed-time sliding-mode gradient flow are given by~\eqref{eq: gf dynamics fx}, with $\lambda_{\mathrm{FxTS}}(x)$ designed via~\eqref{eq:lambda_fxt}, parameters $\alpha=\beta=4$, $p=0.5$, $q=2$, and initial condition $x(0)=[0.4,\,-0.3,\,0.9,\,1.1,\,0.5,\,0.2]^\top$. Under the proposed control law, $h(x)$ converges to zero in fixed time (Fig. \ref{fig:constraint}), after which the dynamics reduce to projected gradient descent on the feasible set, guaranteeing convergence to a KKT point of the AC--OPF problem. The interior-point method in \textit{fmincon} yields $\phi(x^\star)=1.2$, and simulations show convergence of the proposed dynamics to a physically equivalent optimum within $0.4 sec$, significantly faster than PGF (projected gradient flow with $\mu=5$) as shown in Fig. \ref{fig:state}.
\begin{figure}[h!]
    \centering
    \includegraphics[width=1.1\linewidth]{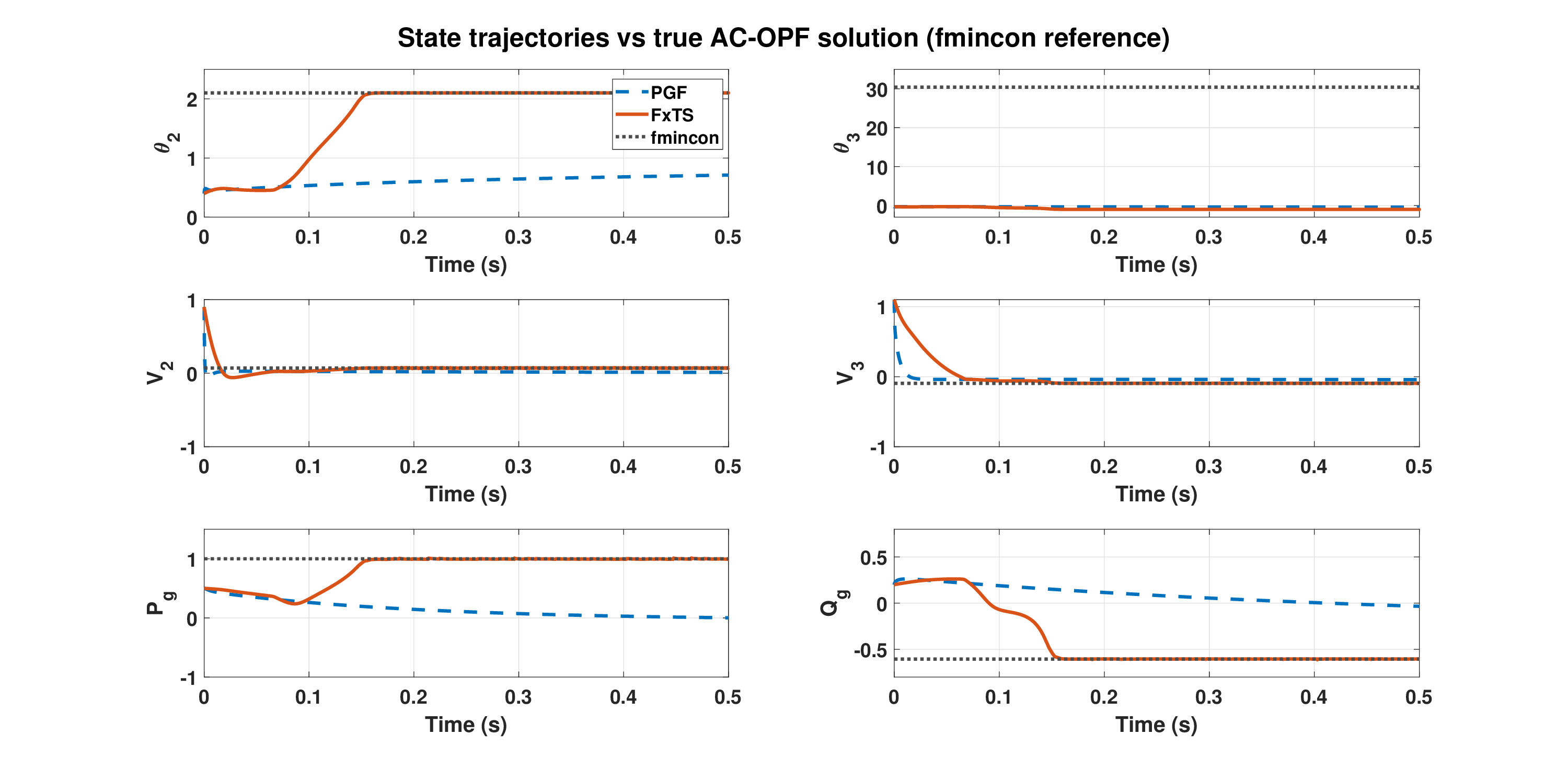}
    \caption{State trajectories of the three-bus AC--OPF system converge to the \textit{fmincon} reference solution $x^\star=[2.1007,\,30.4105,\,0.0698,\,-0.0943,\,1.0000,\,-0.6052]^\top$, with the proposed fixed-time SMC achieving faster convergence than PGF, within $T_c\le 0.4\,\mathrm{s}$; the angle discrepancy is due to the $2\pi$-periodicity of AC power-flow equations.}
    \label{fig:state}
\end{figure}
\subsection{Fixed-Time Convergence for Distributed Estimation}
Consider a network of $N$ agents estimating a common unknown parameter $\theta^\star \in \mathbb{R}^n$. Each agent $i \in \{1,\dots,N\}$ collects local measurements $ y_i = H_i \theta^\star + v_i$, where $H_i \in \mathbb{R}^{m_i \times n}$ is known and $v_i$ is zero-mean Gaussian noise with covariance $R_i \succ 0$. Each agent maintains a local estimate
$\theta_i \in \mathbb{R}^n$, and the aggregate estimation state is defined as  $x = \operatorname{col}(\theta_1,\theta_2,\dots,\theta_N) \in \mathbb{R}^{Nn}$. The distributed estimation problem is formulated as minimization of the convex quadratic cost:
\begin{equation}
    \phi(x)= \frac{1}{2}\sum_{i=1}^N(y_i - H_i \theta_i)^\top R_i^{-1}(y_i - H_i \theta_i),
\end{equation}
whose gradient is
$\nabla \phi(x)= \operatorname{col}\!\left(-H_1^\top R_1^{-1}(y_1 - H_1\theta_1), \dots,-H_N^\top R_N^{-1}(y_N - H_N\theta_N)\right)$. Agents communicate over an undirected connected graph characterized by the graph Laplacian $L \in \mathbb{R}^{N\times N}$. Exact consensus among all local estimates is enforced via the equality constraint $( L\otimes I_n)x=0$, where $\otimes$ denotes the Kronecker product and $I_n$ is the identity matrix.
The sliding variable is given $s(x)=(L \otimes I_n)x$. The constraint Jacobian is $J_h = L \otimes I_n$. To achieve fixed-time consensus and optimal estimation, the continuous-time SMC-based optimization dynamics are designed as \eqref{eq: fxts_1} with control input as \eqref{eq:lambda_fxt} where $\gamma_1=2,\gamma_2=2,\alpha=5,\beta=5, r_1=0.5, p=0.5$, $r_2=1.5,q=1.5$, and $\varepsilon=10^{-6}$.The complete dynamics is given with $G_L=(L^2 \otimes I_n)^{-1}$:
\begin{equation} \label{eq:complete_closed_loop}
    \begin{aligned}
        \dot{x} &= -(I-J_h^\top G_L J_h)F1\\
        &-G_L \Big[\alpha \frac{(L \otimes I_n)x}{(\|(L \otimes I_n)x\|+\varepsilon)^{1-p}}+ \beta \frac{(L \otimes I_n)x}{(\|(L \otimes I_n)x\|+\varepsilon)^{1+q}}\Big]
    \end{aligned}
\end{equation}
where $F_1=\gamma_1 \frac{\nabla \phi(x)}{(\|\nabla \phi(x)\|+\varepsilon)^{1-r_1}}+ \gamma_2 \frac{\nabla \phi(x)}{(\|\nabla \phi(x)\|+\varepsilon)^{1+r_2}}$.
\begin{figure} [ht!]
    \centering
    \includegraphics[width=0.7\linewidth]{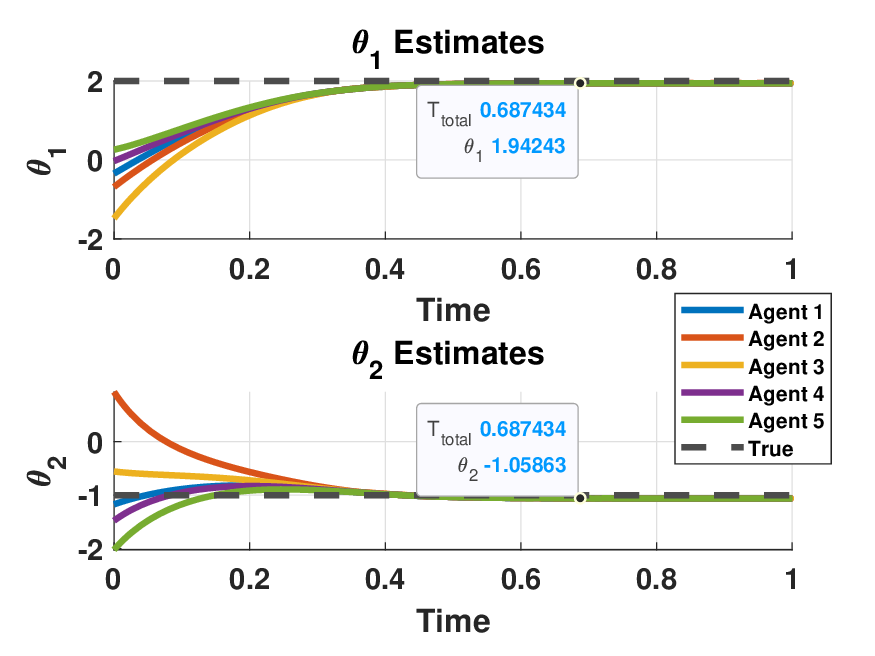}
    \caption{This figure illustrates convergence of distributed parameter estimates in \eqref{eq:complete_closed_loop} to the true values for a network of $N=5$ nodes estimating $\theta_1=2$ and $\theta_2=-1$, with initial condition $x_0=\mathrm{randn}(10,1)$, achieving fixed-time convergence within $T_{total}\le 0.7\,\mathrm{s}$.}
    \label{fig: Distributed estimations}
\end{figure}
We add regularization to prevent numerical singularities and ensure stable implementation when the gradient or consensus error approaches zero. The parameter of each agent converege to true value as shown in Fig. \ref{fig: Distributed estimations} with convergence time of $T_{total} \leq 0.7 sec$.

\section{Conclusion}
This paper proposed a robust fixed-time optimization framework for equality-constrained problems, guaranteeing fixed-time constraint satisfaction and convergence to KKT points for both convex and nonconvex objectives. By modeling constraints as a sliding manifold, the method combines equivalent and fixed-time sliding control terms to ensure optimality and robustness under matched disturbances, as validated through AC–OPF and distributed parameter estimation examples. Future work will address time-varying and stochastic constraints, scalability to large-scale and multi-agent settings, and adaptive or learning-based extensions while preserving fixed-time guarantees.

\bibliography{ref}
\bibliographystyle{IEEEtran}

\end{document}